\documentclass[12pt]{article}
\usepackage{amssymb}
\usepackage{array}
\usepackage{latexsym}
\usepackage{enumerate}
\usepackage{amsmath}
\usepackage{amsfonts}
\usepackage{amsthm}
\usepackage[english]{babel}

\title{\bf{Cayley graph of the symmetric group with generating block transpositions}}
\author{Annachiara~Korchmaros
\\[\baselineskip]
\small{Dipartimento di Matematica e Informatica} \\
\small{Universit\`a di Perugia} \\
\small{06123 Perugia, Italy}}

\newtheorem{thm}{Theorem}[section]
\newtheorem{prop}[thm]{Proposition}

\newtheorem{lem}[thm]{Lemma}
\newtheorem{cor}[thm]{Corollary}

\newtheorem{conj}[thm]{Conjecture}

\theoremstyle{definition}

\newtheorem{rem}[thm]{Remark}

\def\Cay{{\rm{Cay(Sym_n}},T_n)}
\def\Cat{{\rm{\overline{Cay}(Sym_n}},T_n)}

\def\Sym{\mbox{\rm Sym}}

\newcommand{\cupdot}{\mathbin{\mathaccent\cdot\cup}}
\date{}

\begin{document}
\maketitle

\begin{abstract}
This paper deals with the Cayley graph $\Cay,$ where the generating set consists of all block transpositions. A motivation for the study of these particular Cayley graphs comes from current research in Bioinformatics. We prove that ${\rm{Aut}}(\Cay)$ is the product of the right translation group by $\textsf{N}\rtimes \textsf{D}_{n+1},$ where $\textsf{N}$ is the subgroup fixing $T_n$ element-wise and $\textsf{D}_{n+1}$ is a dihedral group of order $2(n+1)$. We conjecture that $\textsf{N}$ is trivial. We also prove that the subgraph $\Gamma$ with vertex-set $T_n$ is a $2(n-2)$-regular graph whose automorphism group is $\textsf{D}_{n+1}$. Furthermore, $\Gamma$ has as many as $n+1$ maximal cliques of size $2.$ Also, its subgraph $\Gamma(V)$ whose vertices are those in these cliques is a $3$-regular, Hamiltonian, and vertex-transitive graph.
\end{abstract}

\rm{\bf{Key words:}} Cayley graph, symmetric group, block transposition, automorphism.

\section{Introduction}\label{s1}
Block transpositions are well-known sorting operations with relevant applications in Bioinformatics; see \cite{FL}. They act on a string by removing a block of consecutive entries and inserting it somewhere else. In terms of the symmetric group $\rm{Sym_n}$ of degree $n$, the strings are identified with the permutations on $[n]=\{1,2,\ldots n\}$, and block transpositions are defined as follows. For any three integers $i,j,k$ with $0\le i <j<k \le n$, the block transposition $\sigma(i,j,k)$ with cut points $(i,j,k)$ turns the permutation  $\pi=[\pi_1\,\cdots\,\pi_n]$
into the permutation $\pi'=[\pi_1\cdots \pi_i\,\,\pi_{j+1}\cdots \pi_k\,\,\pi_{i+1}\cdots \pi_{j}\,\,\pi_{k+1}\cdots \pi_n]$. This action of $\sigma(i,j,k)$ on $\pi$ can also be expressed as the composition $\pi'=\pi\circ\sigma(i,j,k)$.
The set $T_n$ of all block transpositions has size $(n+1)n(n-1)/6$ and
is an inverse closed generating set of $\rm{Sym_n}$. The arising Cayley graph $\Cay$ is a very useful tool since ``sorting a permutation by block transpositions'' is equivalent to finding shortest paths between vertices in $\Cay$; see \cite{EE,FL,Mo}.

Although the definition of a block transposition arose from a practical need, $\Cay$ presents some features, and the most interesting is the existence of automorphisms other than right translations. In addition, the \emph{block transposition graph}, that is the subgraph $\Gamma$ of $\Cay$ with vertex-set $T_n$, has especially nice properties. As we show in this paper, $\Gamma$ is a $2(n-2)$-regular graph whose automorphism group is a dihedral group $\textsf{D}_{n+1}$ of order $2(n+1)$. The definition of $\textsf{D}_{n+1}$ arises from the toric equivalence in $\rm{Sym_n}$ and the reverse permutation. The group $\textsf{D}_{n+1}$ is also an automorphism group of $\Cay$. Therefore, the automorphism group ${\rm{Aut}}(\Cay)$ is the product of the right translation group $R(\Cay)$ by $\textsf{N}\rtimes \textsf{D}_{n+1},$ where $\textsf{N}$ is the subgroup fixing each block transposition. Computer aided exhaustive computation carried out for $n\leq 8$ supports our conjecture that $\textsf{N}$ is trivial, equivalently ${\rm{Aut}}(\Cay)=R(\Cay) \textsf{D}_{n+1}$. We also prove that $R(\Cay) \textsf{D}_{n+1}$ is isomorphic to the direct product of $\Sym_{n+1}$ by a group of order $2$. Furthermore, we show that $\Gamma$ has precisely $n+1$ maximal cliques of size $2$ and look inside the subgraph $\Gamma(V)$ of $\Gamma$ induced by the set $V$ whose vertices are the $2(n+1)$ vertices of these cliques. We prove that $\Gamma(V)$ is $3$-regular. We also observe that $\Gamma(V)$ is Hamiltonian and $\textsf{D}_{n+1}$ is an automorphism group of $\Gamma(V)$ acting transitively (and hence regularly) on $V$. This confirms the Lov\'asz conjecture for $\Gamma(V)$.

For basic facts on Cayley graphs and combinatorial properties of permutations the reader is referred to \cite{AB,B}.

\section{Background on block transpositions}
\label{defs}
Throughout the paper, $n$ denotes a positive integer. In our investigation cases $n\leq 3$ are trivial while case $n=4$ presents some results different from the general case.

For a set $X$ of size $n$, $\Sym_n$ stands for the set of all permutations on $X$. For the sake of simplicity, $[n]=\{1,2,\ldots,n\}$ is usually taken for $X$. As it is customary in the literature on block transpositions, we mostly adopt the functional notation for permutations:  If $\pi\in \rm{Sym}_n$, then $\pi=[\pi_1\pi_2\cdots \pi_n]$ with $\pi(t)=\pi_t$ for every $t\in[n]$, and if
$\pi,\rho\in {\rm{Sym}_n}$ then $\tau=\pi\circ\rho$ is the permutation defined by $\tau(t)=\pi(\rho(t))$ for every $t\in [n]$. The \emph{reverse permutation} is $\omega=[n\,n-1\cdots 1]$, and $\iota=[1\,2\cdots n]$ is the \emph{identity permutation}.

For any three integers, named \emph{cut points}, $(i,j,k)$ with $0\leq i< j< k\leq n$, the \emph{block transposition} (transposition; see \cite{GBH}) $\sigma(i,j,k)$ is defined to be the function on $[n]$:
\begin{equation}\label{feb9}
\sigma(i,j,k)_t =\left\{\begin{array}{ll}
t, &  1\leq t\leq i\quad k+1\leq t\leq n,\\
t+j-i, & i+1\leq t\leq k-j+i,\\
t+j-k, & k-j+i+1\leq t\leq k.
\end{array}
\right.
\end{equation}This shows that $\sigma(i,j,k)_{t+1}=\sigma(i,j,k)_{t} +1$ in the intervals:
\begin{equation}
\label{function}
[1,i],\quad[i+1, k-j+i],\quad[k-j+i+1,k],\quad[k+1, n],
\end{equation}where
\begin{align}
\label{cuppoints}
\sigma(i,j,k)_i&=i;&\sigma(i,j,k)_{i+1}&=j+1;&\sigma(i,j,k)_{k-j+i}&=k;\\
\sigma(i,j,k)_{k-j+i+1}&=i+1;&\sigma(i,j,k)_{k}&=j;&\sigma(i,j,k)_{k+1}&=k+1.\notag
\end{align}Actually, $\sigma(i,j,k)$ can also be represented as the permutation
\begin{equation}
\label{eq22ott12}
\sigma(i,j,k)=\left\{\begin{array}{ll}
[1\cdots i\,\, j+1\cdots k\,\, i+1\cdots j\,\, k+1 \cdots n], & 1\leq i,\,k< n,\\
{[j+1\cdots k\,\, 1\cdots j\,\, k+1 \cdots n]}, & i=0,\,k< n,\\
{[1\cdots i\,\,j+1\cdots n\,\, i+1\cdots j]}, & 1\leq i,\,k=n,\\
{[j+1\cdots n\,\, 1\cdots j]}, & i=0,\,k=n
\end{array}
\right.
\end{equation}such that the action of $\sigma(i,j,k)$ on $\pi$ is defined as the product
$$\pi\circ\sigma(i,j,k)=[\pi_1\cdots \pi_i\,\,\pi_{j+1}\cdots \pi_k\,\,\pi_{i+1}\cdots \pi_{j}\,\,\pi_{k+1}\cdots \pi_n].$$ Therefore, applying a block transposition on the right of $\pi$ consists in switching two adjacent subsequences of $\pi$, namely \emph{blocks}, without changing the order of the integers within each block. This may also be expressed by $$[\pi_1\cdots\pi_i|\pi_{i+1}\cdots\pi_j|\pi_{j+1}\cdots\pi_k|\pi_{k+1}\cdots\pi_n].$$

{}From now on, $T_n$ denotes the set of all block transpositions on $[n]$. The size of $T_n$ is equal to $n(n+1)(n-1)/6$. Obviously, $T_n$ is not a subgroup of $\rm{Sym_n}$. Nevertheless,
$T_n$ is power and inverse closed. For any cut points $(i,j,k)$,
\begin{equation}
\label{eqa18oct}
\sigma(i,j,k)^{-1}=\sigma(i,k-j+i,k),\qquad\sigma(i,i+1,k)^{j-i}=\sigma(i,j,k).
\end{equation}Also, for any two integers $i,\,k$ with $0\le i<k\le n$ the subgroup generated by $\sigma(i,i+1,k)$ consists of all $\sigma(i,j,k)$ together with the identity. In particular, $\sigma(0,1,n)$ generates a subgroup of order $n$ that often appears in our arguments. Throughout the paper, $\beta=\sigma(0,1,n)$ and $B$ denotes the set of nontrivial elements of the subgroup generated by $\beta$.

We introduce some subsets in $T_n$ that plays a relevant role in our study. Every permutation $\bar{\pi}$ on $[n-1]$ extends to a permutation $\pi$ on $[n]$ such that $\pi_t=\bar{\pi}_t$ for $1\le t \le n-1$ and $\pi_n=n$. Hence,
$S_{n-1}$ is naturally embedded in $T_n$ since every $\sigma(i,j,k)\in T_n$ with $k\neq n$ is identified with the block transposition $\bar\sigma(i,j,k)$. On the other side, every permutation $\pi'$ on $\{2,3,\ldots,n\}$ extends to a permutation on $[n]$ such that $\pi_t=\pi'_t,$ for $2\le t \le n$ and $\pi_1=1$. Thus, $\sigma(i,j,k)\in T_n$ with $i\neq 0$ is identified with the block transposition $\sigma'(i,j,k)$. The latter block transpositions form the set
$$ S_{n-1}^\triangledown=\{ \sigma(i,j,k)|\,i\neq 0\}.$$ Also, $$S_{n-2}^\vartriangle=S_{n-1}\cap S_{n-1}^\triangledown$$ is the set of all block transpositions on the set $\{2,3,\ldots,n-1\}$.
Our discussion leads to the following results.
\begin{lem}[Partition lemma] Let $L=S_{n-1}\setminus S_{n-2}^\vartriangle$ and let $F=S_{n-1}^\triangledown\setminus S_{n-2}^\vartriangle$. Then
\label{lem1oct9}
$$T_n=B\cupdot L \cupdot F \cupdot S_{n-2}^\vartriangle.$$
\end{lem}With the above notation, $L$ is the set of all $\sigma(0,j,k)$ with $k\neq n$, and $F$ is the set of all $\sigma(i,j,n)$ with $i\neq0$. Furthermore, $|B|=n-1,\,|L|=|F|=(n-1)(n-2)/2,$ and $|S_{n-2}^\vartriangle|=(n-1)(n-2)(n-3)/6.$

\section{Toric equivalence in the symmetric group}
The definition of toric (equivalence) classes in $\rm{Sym_n}$ requires to consider permutations on $[n]^0=\{0,1,\ldots,n\}$ and recover the permutations $\pi=[\pi_1\,\cdots\pi_n]$ on $[n]$ in the form $[0\,\pi]$, where $[0\,\pi]$ stands for the permutation $[\pi_0\,\pi_1\,\cdots \pi_n]$ on $[n]^0$ with $\pi_0=0$. Let $$\alpha=[1\,2\,\ldots\, n\,0].$$For any integer $r$ with $0\le r \le n$,
\begin{equation}\label{alpha_powers}
\alpha^r_x\equiv x+r {\pmod{n+1}},\qquad 0\leq x\leq n.
\end{equation}This gives rise to \emph{toric maps} $\textsf{f}_r$ on $\rm{Sym_n}$ with $0\le r\le n,$ defined by
\begin{equation}\label{eq2oct9}
\textsf{f}_r(\pi)=\rho\Longleftrightarrow[0\rho]= \alpha^{n+1-\pi_r}\circ [0\,\pi] \circ \alpha^r.
\end{equation}
The \emph{toric class} of $\pi$ is
\begin{equation}\label{eq1oct10}
\textsf{F}(\pi)=\{\textsf{f}_r(\pi)|r=0,1,\ldots,n\}.
\end{equation}
Since
\begin{equation}
\label{eq9oct}
(\textsf{f}_r(\pi))_t=\pi_{r+t}-\pi_r,\qquad t\in [n],
\end{equation}where the indices are taken mod$(n+1)$, (\ref{eq1oct10}) formalizes the intuitive definition of toric classes introduced in \cite{EE} by Eriksson and his coworkers.

In general, the toric class of $\pi$ comprises $n+1$ permutations, but it may consist of a smaller number of permutations and can even collapse to a unique permutation. The latter case occurs when $\pi$ is the identity permutation or the reverse permutation. The number of elements in a toric class is always a divisor of $n+1$, and there are exactly $\varphi(n+1)$ classes that have only one element, where $\varphi$ is the Euler function; see \cite{C}.

{}From (\ref{eq2oct9}), $\textsf{f}_s\circ \textsf{f}_r=\textsf{f}_{s+r}$, where the indices are taken mod$(n+1)$. In fact, $\textsf{f}_s\circ\textsf{f}_r(\pi)=\textsf{f}_s(\varphi)$ with $[0\,\varphi]=\alpha^{n+1-\pi_r}\circ[0\,\pi]\circ\alpha^r$. Also, $\textsf{f}_s\circ\textsf{f}_r(\pi)=\mu$ with
$$\begin{array}{lll}
[0\,\mu]&=& \alpha^{n+1-\varphi_s}\circ [0\,\varphi] \circ \alpha^s\\
{}&=&\alpha^{n+1-\pi_{r+s}}\circ [0\,\pi] \circ\alpha^{r+s},
\end{array}$$where $\varphi_s=\pi_{r+s}-\pi_r$. Hence, $\textsf{f}_r=\textsf{f}^{\,r}$ with $\textsf{f}=\textsf{f}_1$, and the set
$$\textsf{F}=\{\textsf{f}_r|r=0,1,\ldots, n\}$$ is a cyclic group of order $n+1$ generated by \textsf{f}.

If $\pi\in \rm{Sym_n}$ and $0\le r \le n$, then
\begin{equation}
\label{lem3oct11}
\textsf{f}^{-1}_{r}(\pi)=\textsf{f}_{\pi_r}(\pi^{-1}).
\end{equation}
In particular, $\textsf{f}^{-1}_{r}(\pi)=\textsf{f}_{r}(\pi^{-1})$ provided that $\pi_r=r$.

The \emph{reverse map} $\textsf{g}$ on $\rm{Sym_n}$ is defined by
\begin{equation}
\label{eq2oct9b}
\textsf{g}(\pi)=\rho\Longleftrightarrow[0\,\rho]=[0\,w]\circ [0\,\pi] \circ [0\,w].
\end{equation}
$\textsf{g}$ is an involution, and
\begin{equation}
\label{eq9octb}
(\textsf{g}(\pi))_t=n+1-\pi_{n+1-t}\qquad\mbox{ for every }t\in [n].
\end{equation}
Also, for every integer $r$ with $0\le r\le n$,
\begin{equation}
\label{eqoct15a}
\textsf{g}\circ\textsf{f}_r \circ \textsf{g}=\textsf{f}_{n+1-r}.
\end{equation}In fact, $\textsf{g}\circ\textsf{f}_r \circ \textsf{g}(\pi)=\textsf{g}\circ\textsf{f}_r(\rho)=\textsf{g}(\rho')$ with
$[0\,\rho]=[0\,w]\circ[0\,\pi]\circ[0\,w]$, and $[0\,\rho']=\alpha^{n+1-\rho_r}\circ [0\,\rho] \circ \alpha^r$. Also, $\textsf{g}\circ\textsf{f}_r \circ \textsf{g}(\pi)=\mu$ with
$$\begin{array}{lll}
[0\,\mu]&=& [0\,w]\circ [0\,\rho'] \circ [0\,w]\\
{}&=&[0\, w]\circ\alpha^{n+1-\rho_{r}}\circ [0\, w]\circ [0\,\pi]\circ [0\, w] \circ\alpha^{r}\circ[0\, w]\\
{}&=&\alpha^{n+1-\pi_{n+1-r}}\circ [0\,\pi]\circ\alpha^{n+1-r}
\end{array}$$since $[0\,w]\circ\alpha^r \circ [0\,w]=\alpha^{n+1-r}$.

Now, we show that the toric maps and reverse map take any block transposition into a block transposition, as corollary of the following lemma.
\begin{lem}
\label{lemc19ott2013}
Let $\sigma(i,j,k)$ be any block transposition on $[n]$. Then
\begin{equation}
\label{22march2015A}
\emph{\textsf{f}}(\sigma(i,j,k))=\left\{\begin{array}{ll}
\sigma(i-1,j-1,k-1), & i >0 ,\\
\sigma(k-j-1,n-j,n), & i=0.
\end{array}
\right.
\end{equation}and
\begin{equation}\label{eq1oct11}
\emph{\textsf{g}}(\sigma(i,j,k))=\sigma(n-k,n-j,n-i)
\end{equation}
\end{lem}
\begin{proof}
For $i>0$, let $\sigma=\sigma(i,j,k)$. Then (\ref{eq2oct9}) reads
$$
\textsf{f}(\sigma)=\rho\Longleftrightarrow[0\rho]= \alpha^{n}\circ [0\,\sigma] \circ \alpha,
$$as $\sigma_1=1$ from (\ref{eq22ott12}). Since $\alpha_t=t+1\pmod{n+1}$ for every $t\in [n]^0$, then, by (\ref{feb9}),
\begin{equation}\label{feb91}
\sigma\circ\alpha_t =\left\{\begin{array}{ll}
t+1, &  0\leq t+1\leq i\quad k+1\leq t\leq n,\\
t+j-i+1, & i+1\leq t+1\leq k-j+i,\\
t+j-k, & k-j+i+1\leq t+1\leq k.\\
\end{array}
\right.
\end{equation}In addition, $\alpha^n_t=\alpha^{-1}_t=t-1\pmod{n+1}$ for every $t\in [n]^0$. Thus, (\ref{feb91}) yields
$$
\alpha^n\circ\sigma\circ\alpha_t =\left\{\begin{array}{ll}
t, &  0\leq t\leq i-1\quad k\leq t\leq n,\\
t+j-i, & i\leq t\leq k-j+i-1,\\
t+j-k-1, & k-j+i\leq t\leq k-1,\\
\end{array}
\right.
$$hence the statement holds true for $i>0$, by (\ref{feb9}). Now, suppose $i=0$. Then, from (\ref{eq22ott12}) follows $\sigma_1=j+1$, and (\ref{eq2oct9}) reads
$$
\textsf{f}(\sigma(0,j,k))=\rho\Longleftrightarrow[0\rho]= \alpha^{n-j}\circ [0\,\sigma(0,j,k)] \circ \alpha.
$$Replacing $i$ with $0$ in (\ref{feb91}),
$$\alpha^{n-j}\circ\sigma(0,j,k)\circ\alpha_t =\left\{\begin{array}{ll}
t, &  0\leq t\leq k-j-1,\\
t-k, & k-j\leq t\leq k-1,\\
t-j, & k\leq t\leq n.\\
\end{array}
\right.$$Therefore, the assertion for $\textsf{f}$ follows from (\ref{feb9}).

For the reverse map, (\ref{eq22ott12}) yields (\ref{eq1oct11}).
\end{proof}
Since $T_n$ is a power closed subset of $\Sym_n$ and toric maps are powers of $\textsf{f}$, Lemma \ref{lemc19ott2013} has the following corollary.
\begin{cor}
\label{th1}
$T_n$ is invariant under the action of toric maps and the reverse map.
\end{cor}

\section{Structure of Cayley graphs on symmetric groups whose connection set consists of all block transpositions}
\label{cayleyleft}
Since $T_n$ is an inverse closed generator set of $\rm{Sym_n}$ which does not contain the identity permutation $\iota$, the (left-invariant) Cayley graph $\Cay$ is an undirected connected simple graph, where $\{\pi,\rho\}$ is an edge if and only if $\rho=\sigma(i,j,k)\circ \pi,$ for some $\sigma(i,j,k)\in T_n$.

By a result of Cayley, every $h\in \rm{Sym_n}$ defines a \emph{right translation} $\textsf{h}$ which is the automorphism  of $\Cay$ that takes the vertex $\pi$ to the vertex $\pi\circ h$, and hence the edge $\{\pi,\rho\}$ to the edge $\{\pi\circ h,\rho\circ h\}$. These automorphisms form the \emph{right translation group} $R(\Cay)$ of $\Cay$. Clearly, ${\Sym}_n\cong R(\Cay)$. Furthermore,
since $R(\Cay)$ acts regularly on ${\Sym}_n$, every automorphism of $\Cay$ is the product of a right translation by an automorphism fixing $\iota$.

One may ask if there is a nontrivial automorphism of $\Cay$ fixing $\iota$. The answer is affirmative by the following results.
\begin{lem}\label{mat27oct}
For every $\pi,\rho \in {\Sym}_n,$
\begin{itemize}
\item[\rm(i)] $\emph{\textsf{f}}_r(\rho\circ\pi)=\emph{\textsf{f}}_{\pi_r}(\rho)\circ\emph{\textsf{f}}_r(\pi);$
\item[\rm(ii)] $\emph{\textsf{g}}(\rho\circ\pi)=\emph{\textsf{g}}(\rho)\circ\emph{\textsf{g}}(\pi).$
\end{itemize}
\end{lem}
\begin{proof}
(i) From (\ref{eq2oct9}), $\textsf{f}_r(\rho\circ\pi)=\mu$ with
$$\begin{array}{lll}
[0\,\mu]&=& \alpha^{n+1-(\rho\circ\pi)_r}\circ [0\,\rho] \circ [0\,\pi]\circ \alpha^r\\
{}&=&\alpha^{n+1-(\rho\circ\pi)_r}\circ [0\,\rho]\circ\alpha^{\pi_r} \circ\alpha^{n+1-\pi_r} [0\,\pi]\circ \alpha^r.
\end{array}$$Now, the first assertion follows from (\ref{eq2oct9}).

(ii) A similar argument depending on (\ref{eq2oct9b}) shows that the second assertion holds true for $\textsf{g}$.
\end{proof}

\begin{prop}
\label{prop10oct} Toric maps and the reverse map are automorphisms of $\Cay$.
\end{prop}
\begin{proof}Let $\pi,\rho \in {\Sym}_n$ be any two adjacent vertices of $\Cay$. Then, $\rho=\sigma\circ\pi$ for some $\sigma=\sigma(i,j,k)\in T_n.$ Here, Lemma \ref{mat27oct} yields $\textsf{f}(\rho)=\textsf{f}_{\pi_1}(\sigma)\circ\textsf{f}(\pi).$ Therefore, the assertion for $\textsf{f}$ follows from Proposition \ref{th1}. By induction on $r$, this holds true for all toric maps.

A similar argument can be used to show the assertion for the reverse map.
\end{proof}
By (\ref{eqoct15a}), the set consisting of $\textsf{F}$ and its coset $\textsf{F}\circ\textsf{g}$ is a dihedral group $\textsf{D}_{n+1}$ of order $2(n+1)$. Clearly, $\textsf{D}_{n+1}$ fixes $\iota$. Now, Proposition \ref{prop10oct} has the following corollary.
\begin{cor}
\label{teor1} The automorphism group of $\Cay$ contains a dihedral subgroup of order $2(n+1)$ fixing the identity permutation.
\end{cor}
{}From now on, the term of \emph{toric-reverse group} stands for $\textsf{D}_{n+1}$, and $\textsf{G}$ denotes the stabilizer of $\iota$ in the automorphism group of $\Cay$. By Corollary \ref{teor1}, the question arises of whether or not $\textsf{D}_{n+1}$ is already $\textsf{G}$. We state our result on this problem.

Clearly, $\textsf{G}$ preserves the subgraph
of $\Cay$ whose vertices are the block transpositions. We call this subgraph $\Gamma$ the \emph{block transposition graph} and use $\textsf{R}$ to denote its automorphism group. The kernel of the permutation representation of $\textsf{G}$ on $T_n$ is a normal subgroup $\textsf{N}$, and the factor group $\textsf{G}/\textsf{N}$ is a subgroup of $\textsf{R}$. Since $N$ is the subgroup fixing $T_n$ element-wise, $\textsf{D}_{n+1}$ and $\textsf{N}$ have trivial intersection, and the toric-reverse group can be regarded as a subgroup of $\textsf{G}/\textsf{N}$. Our main result in this paper is a proof of the theorem below.
\begin{thm}
\label{teor2} The automorphism group of $\Gamma$ is the toric-reverse group.
\end{thm}The proof of Theorem \ref{teor2} will be completed in Section \ref{s6}, using several results on combinatorial properties of $\Gamma$, especially on the set of its maximal cliques of size $2$ in Section \ref{s5}. As a corollary, $\textsf{G}=\textsf{N}\rtimes\textsf{D}_{n+1}$. From this the following result is obtained.
\begin{thm}
\label{teor2b} The automorphism group of $\Cay$ is the product of the right translation group by $\emph{\textsf{N}}\rtimes \emph{\textsf{D}}_{n+1}.$
\end{thm}
\begin{rem} Computation shows that $\textsf{N}$ is trivial for $n\leq 8$. This motivates the following conjecture.
\begin{conj} \label{March24}
The automorphism group of $\Cay$ is the product of the right translation group by the toric-reverse group.
\end{conj}
\end{rem}
In order to prove the Conjecture \ref{March24}, a useful result is the following proposition, where $R(\Cay) \textsf{D}_{n+1}$ is the set of all products $\textsf{d}\circ\textsf{h}$ with $\textsf{d}\in\textsf{D}_{n+1}$ and $\textsf{h}\in R(\Cay)$.
\begin{prop}
\label{29oct}
The product of the right translation group by the toric-reverse group is isomorphic to the direct product of ${\Sym}_{n+1}$ by a group of order $2$.
\end{prop}
\begin{proof} First we find an involution $\textsf{t}$ that centralizes $R(\Cay)\textsf{F}.$ Two automorphisms of $\Cay$ arise from the reverse permutation $w$, namely $\textsf{g}$ and the right translation $\textsf{w},$ and $\textsf{g}\circ\textsf{w}$ is the automorphism $\textsf{t}$ that takes $\pi$ to $w\circ \pi$. Obviously, $\textsf{t}\in R(\Cay)\textsf{D}_{n+1}$ is an involution as $\textsf{g}$ and $\textsf{w}$ are involutions.

Now, we show that $\textsf{t}$ centralizes $R(\Cay)\textsf{F}$. In order to do that, it suffices to prove $\textsf{t}\circ\textsf{f}=\textsf{f}\circ\textsf{t}$ since $\textsf{t}$ commutes with any right translation. Now, for every $\pi\in {\Sym}_n,$
$$\textsf{f}\circ\textsf{t}(\pi)=\rho\Longleftrightarrow [0\,\rho]=\alpha^{\pi_1}\circ[0\, w]\circ\alpha^n\circ[0\,\pi]\circ\alpha,$$by Lemma \ref{mat27oct} (i). On the other hand,
$$\textsf{t}\circ\textsf{f}(\pi)=\rho'\Longleftrightarrow [0\,\rho']=[0\, w]\circ \alpha^{n-\pi_1}\circ[0\,\pi]\circ\alpha.$$Since $[0\, w]\circ\alpha^{n+1-\pi_{1}}=\alpha^{\pi_1}\circ[0\, w]$ by (\ref{eqoct15a}), this yields
that $\textsf{t}$ commutes with $\textsf{F}$.

Now, we show that $\textsf{t}$ is off $R(\Cay)\textsf{F}$. Suppose on the contrary that there exists some right translation $\textsf{h}$ such that $\textsf{t}=\textsf{h}\circ\textsf{f}^{\,r}$ with $0\le r \le n$. Since $\textsf{t}$ is an involution this implies $\textsf{t}\circ\textsf{h}\in\textsf{F},$  and then $\textsf{t}\circ\textsf{h}$ fixes $\iota$. On the other hand, $\textsf{t}\circ\textsf{h}(\iota)=w\circ h.$ Thus, $h=w$ is an involution. Therefore, $\textsf{t}\circ\textsf{h}=\textsf{t}\circ\textsf{w}$ is an involution as well. Since $\iota$ is the only involution in $\textsf{F},$ $\textsf{t}\circ\textsf{h}=\iota,$ hence $\textsf{t}=\textsf{h}.$ Thus, we have proven that $\textsf{t}$ is a right translation. Since the center of $R(\Cay)$ is trivial while $\textsf{t}$ commutes with any right translation, we have $\textsf{t}\not\in R(\Cay),$ a contradiction.

Therefore, $\textsf{t}\in R(\Cay)\textsf{D}_{n+1}=R(\Cay)\textsf{F}\times \langle\textsf{t}\rangle.$ In fact, a straightforward computation shows that $$\textsf{h}\circ\textsf{f}^{\,r}\circ\textsf{g}=\textsf{h'}\circ\textsf{f}^{\,r}\circ\textsf{g}\circ\textsf{w},$$where $\textsf{h}=\textsf{h'}\circ\textsf{w},$ for any right translation $\textsf{h}$ and $0\le r\le n.$

To prove the isomorphism $R(\Cay)\textsf{F}\cong \Sym_{n+1}$, let $\Phi$ be the map that takes $\textsf{h}\circ\textsf{f}^{\,r}$ to $[0\,h^{-1}]\circ \alpha^{n+1-r}$. For any $h,k,\pi\in \Sym_n$ and $0\leq r,u \le n$, by Lemma \ref{mat27oct} (i),
$$\textsf{h}\circ\textsf{f}_r\circ \textsf{k}\circ\textsf{f}_u(\pi)=\textsf{h}\circ\textsf{f}_r(\textsf{f}_u(\pi)\circ k)=\textsf{f}_{u+k_r}(\pi)\circ\textsf{f}_{r}(k)\circ h.
$$This shows that $\textsf{h}\circ\textsf{f}_r\circ \textsf{k}\circ\textsf{f}_u=\textsf{d}\circ \textsf{f}_{u+k_r}$ with $d=\textsf{f}_r(k) \circ h$. Then,
$$\begin{array}{lll}
\Phi(\textsf{h}\circ\textsf{f}_r\circ \textsf{k}\circ\textsf{f}_u)&=&[0\,h^{-1}]\circ[0\,\textsf{f}_r(k)^{-1}]\circ\alpha^{n+1-u-k_r}\\
{}&=&[0\,h^{-1}]\circ\alpha^{n+1-r}\circ[0\,k^{-1}]\circ\alpha^{n+1-u}
\end{array}$$since
$[0\,\textsf{f}_r(k)^{-1}]=[0\,\textsf{f}_{k_r}(k^{-1})]=\alpha^{n+1-r}\circ[0\,k^{-1}]\circ\alpha^{k_r},$ by (\ref{lem3oct11}). On the other hand,
$$\Phi(\textsf{h}\circ\textsf{f}_r)\circ\Phi(\textsf{k}\circ\textsf{f}_u)=[0\,h^{-1}]\circ\alpha^{n+1-r}\circ[0\,k^{-1}]\circ\alpha^{n+1-u}.$$Hence, $\Phi$ is a group homomorphism from $R(\Cay)\textsf{F}$ into the symmetric group on $[n]^0$. Furthermore, $\ker(\Phi)$ is trivial. In fact,  $[0\,h^{-1}]\circ \alpha^{n+1-r}=[0\,\iota]$ only occurs for $h=\iota$ since the inverse of $\alpha^{n+1-r}$ is the permutation $\alpha^{r}$ not fixing $0.$ This together with $|R(\Cay)\textsf{F}|=(n+1)!$ shows that $\Phi$ is bijective.
\end{proof}

\section{Combinatorial properties of \\the block transposition graph}\label{s5}
In this section we refer to $\Cat$ as the (right-invariant) Cayley graph, where $\{\pi,\rho\}$ is an edge if and only if $\rho=\pi \circ\sigma(i,j,k),$ for some $\sigma(i,j,k)\in T_n$. This is admissible since the the left-invariant and right-invariant Cayley graphs are isomorphic. In fact, the map taking any permutation to its inverse is such an isomorphism. Our choice is advantageous as the proofs in this section are formally simpler with the right-invariant Cayley graph notation. This change may be justified by (\ref{feb9}), which shows that computing $\pi\circ\sigma$ is more natural and immediate than $\sigma\circ\pi,$ whenever $\pi\in\Sym_n$ and $\sigma\in T_n$.

Now, every toric map $\textsf{f}_r$ is replaced by $\bar{\textsf{f}}_r$ defined as
\begin{equation}
\label{March29+}
\bar{\textsf{f}}_r(\pi)=(\textsf{f}_r(\pi^{-1}))^{-1},\qquad \pi\in \Sym_n.
\end{equation}
In addition, from (\ref{lem3oct11}) applied to $r=1$,
\begin{equation}
\label{March29}
\bar{\textsf{f}}(\pi)=\textsf{f}(\pi)^{\pi^{-1}_1},\qquad \pi\in \Sym_n.
\end{equation}
This shows that $\bar{\textsf{f}}\not\in \textsf{F}$. Nevertheless, $\bar{\textsf{f}}_r=\bar{\textsf{f}}^{\,r}$, as $\textsf{f}_r=\textsf{f}^{\,r}$ for any integer $r$ with $0\leq r\leq n$. Then $\bar{\textsf{F}}\cong \textsf{F}$, where
$\bar{\textsf{F}}$ is the group generated by $\bar{\textsf{f}}$, and the natural map ${\bar{\textsf{f}}}_{r}\rightarrow \textsf{f}_{r}$ is an isomorphism. 

Furthermore, since $\textsf{g}(\pi^{-1})^{-1}=\textsf{g}(\pi)$ for any $\pi\in \Sym_n$, $\bar{\textsf{g}}$ coincides with $\textsf{g}$. In addition, the group $\overline{\textsf{D}}_{n+1}$ generated by $\bar{\textsf{f}}$ and
$\textsf{g}$ is isomorphic to $\textsf{D}_{n+1}$, and then this is the \emph{toric-reverse group} of $\Cat$. 
\begin{lem}
\label{22marchC2015}
Let $\sigma(i,j,k)$ be any block transposition on $[n]$. Then
\begin{equation}
\label{22march2015}
\bar{\emph{\textsf{f}}}(\sigma(i,j,k))=\left\{\begin{array}{ll}
\sigma(i-1,j-1,k-1), & i >0 ,\\
\sigma(j-1,k-1,n), & i=0.
\end{array}
\right.
\end{equation}
\end{lem}
\begin{proof} 
Let $\sigma=\sigma(i,j,k)$. For $i>0$, we obtain $\sigma_1=1$ from (\ref{eq22ott12}). Therefore,  $\bar{\textsf{f}}(\sigma)=\textsf{f}(\sigma)$ by (\ref{March29}). Hence the statement for $i>0$ follows from Lemma \ref{lemc19ott2013}.

Now, suppose $i=0$. By (\ref{March29+}) and Lemma \ref{lemc19ott2013},
$$\bar{\textsf{f}}(\sigma)=(\textsf{f}(\sigma^{-1}))^{-1}=(\textsf{f}(\sigma(0,k-j,k)))^{-1}=\sigma(j-1,n-(k-j),n)^{-1}$$
which is equal to $\sigma(j-1,k-1,n)$, by (\ref{eqa18oct}). Therefore, the statement also holds for $i=0$. 
\end{proof}
Now, we transfer our terminology from Section \ref{cayleyleft}. In particular, $\bar{\textsf{f}}$ and its powers are the \emph{toric maps},
$\bar{\textsf{F}}$ the \emph{toric group}, and $\bar{\Gamma}$ is the \emph{block transposition graph} of $\Cat$.
\begin{prop}
\label{22march22D2015} Toric maps and the reverse map are automorphisms of $\Cat$.
\end{prop}
\begin{proof}{}From Lemma \ref{mat27oct} (ii) and Corollary \ref{th1} follows that the reverse map $\textsf{g}$ is also an automorphism of $\Cat$.
 
Now, it suffices to prove the claim for $\bar{\textsf{f}}.$ Take an edge $\{\pi,\rho\}$ of $\Cat$. Then $\rho=\pi\circ\sigma$ with $\sigma\in T_n$, and
$$\bar{\textsf{f}}(\pi\circ \sigma)=(\textsf{f}(\sigma^{-1}\circ \pi^{-1}))^{-1}=(\textsf{f}_{\pi_1^{-1}}(\sigma^{-1})\circ\textsf{f}(\pi^{-1}))^{-1}=
\bar{\textsf{f}}(\pi)\circ\textsf{f}_{\pi_1^{-1}}(\sigma^{-1})^{-1},$$by Lemma \ref{mat27oct} (i). Here $\textsf{f}_{\pi_1^{-1}}(\sigma^{-1})^{-1}\in T_n$ since $T_n$ is inverse closed, by (\ref{eqa18oct}), and $\textsf{F}$ leaves $T_n$ invariant, by Corollary \ref{th1}. Therefore, $\bar{\textsf{f}}(\pi)$ and $\bar{\textsf{f}}(\rho)$ are incident in $\Cat$.
\end{proof}
As consequence of Proposition \ref{22march22D2015}, all the results in Section \ref{cayleyleft} hold true for $\Cat$ up to the obvious change ``right-translation'' to ``left-translation''.

Now, we state some results on the components of the partition in Lemma \ref{lem1oct9}.

Since $B$ consists of all nontrivial elements of a subgroup of $T_n$ of order $n$, the block transpositions in $B$ are the vertices of a complete graph  of size $n-1$. Lemma \ref{lem1oct9} and (\ref{eq1oct11}) give the following property.
\begin{cor}\label{lem2oct11}
The reverse map preserves both $B$ and $S_{n-2}^\vartriangle$ while it switches $L$ and $F$.
\end{cor}
\begin{lem}\label{lem1oct11}
No edge of $\Cat$ has one endpoint in $B$ and the other in $S_{n-2}^\vartriangle$.
\end{lem}
\begin{proof}
Suppose on the contrary that $\{\sigma(i',j',k'),\sigma(0,j,n)\}$ with $i'\neq 0$ and $k'\neq n$ is an edge of $\Cat$. By (\ref{eqa18oct}),
$\rho=\sigma(0,n-j,n)\circ \sigma(i',j',k')\in T_n$. Also, $\rho\in B$ as $\rho_1\neq 1$ and $\rho_n\neq n$. Since $B$ together with the identity is a group, $\sigma(0,j,n)\circ\rho$ is also in $B$. This yields $\sigma(i',j',k')\in B$, a contradiction with Lemma \ref{lem1oct9}.
\end{proof}
The proofs of the subsequent properties use a few more equations involving block transpositions which are stated in the following two lemmas.
\begin{lem}\label{lem2oct13}
In each of the following cases $\{\sigma(i,j,k),\sigma(i',j',k')\}$ is an edge of $\Cat$.
\begin{itemize}
\item[\rm(i)]$(i',j')=(i,j);$
\item[\rm(ii)]$(i',j')=(j,k)$ for $k<k';$
\item[\rm(iii)]$(j',k')=(j,k)$;
\item[\rm(iv)]$(j',k')=(i,j)$ for $i'<i;$
\item[\rm(v)]$(i,k)=(i',k')$ for $j<j'$.
\end{itemize}
\end{lem}
\begin{proof}
(i) W.l.g. $k'<k$. By (\ref{eq22ott12}), $\sigma(i,j,k)=\sigma(i,j,k')\circ\sigma(k'-j+i,k',k).$
(iii) W.l.g. $i'<i$. From (\ref{eq22ott12}),
$\sigma(i,j,k)=\sigma(i',j,k)\circ\sigma(i',k-j+i',k-j+i).$

In the remaining cases, from (\ref{eq22ott12}),
$$\begin{array}{lll}
\sigma(i,j,k)=\sigma(j,k,k')\circ\sigma(i,k'-k+j,k'),\\
\sigma(i,j,k)=\sigma(i',i,j)\circ\sigma(i',j-i+i',k),\\
\sigma(i,j,k)=\sigma(i,j',k)\circ\sigma(i,k-j+j',k).
\end{array}$$Hence the statements hold.\end{proof}
The proof of the lemma below is straightforward and requires only (\ref{eq22ott12}).
\begin{lem}
\label{prop1oct12}
The following equations hold.
\begin{itemize}
\item[\rm(i)]$\sigma(i,j,n)=\sigma(0,j,n)\circ\sigma(0,n-j,n-j+i)$ for $i\neq 0;$
\item[\rm(ii)]$\sigma(i,j,n)=\sigma(0,i,j)\circ\sigma(0,j-i,n)$ for $i\neq 0;$
\item[\rm(iii)]$\sigma(0,j,n)=\sigma(i,j,n)\circ\sigma(0,i,n-j+i);$
\item[\rm(iv)]$\sigma(0,j,n)=\sigma(0,j,j+i)\circ\sigma(i,j+i,n)$ for $i\neq 0.$
\end{itemize}
\end{lem}
\begin{lem}
\label{leaoct19} Let $i$ be an integer with $0<i\leq n-2$.
\begin{itemize}
\item[\rm(i)] If $\sigma(i,j,n)=\sigma(0,\bar{j},n)\circ\sigma(i',j',k'),$ then $\bar{j}=j.$
\item[\rm(ii)] If $\sigma(i,j,n)=\sigma(i',j',k')\circ \sigma(0,\bar{j},n),$ then $\bar{j}=i-j.$
\end{itemize}
\end{lem}
\begin{proof} (i) Assume $\bar{j}\neq j.$ From Lemma \ref{prop1oct12} (i) and (\ref{eqa18oct}),
\begin{equation}\label{matoct25}
\sigma(i',j',k')=\sigma(0,j^*,n)\circ\sigma(0,n-j,n-j+i),
\end{equation}where $j^*$ denotes the smallest positive integer such that $j^*\equiv j-\bar{j} \pmod n$. First we prove $i'=0$. Suppose on the contrary, then $$(\sigma(0,j^*,n)\circ\sigma(0,n-j,n-j+i))_1=1.$$On the other hand, $\sigma(0,n-j,n-j+i)_1=n-j+1$ and $\sigma(0,j^*,n)_{n-j+1}=n-\bar{j}+1$ since
$\sigma(0,j^*,n)_t=t+j^*\pmod n$ by (\ref{feb9}). Thus, $n-\bar{j}+1=1,$ a contradiction since $\bar{j}<n$.

Now, from (\ref{matoct25}), $\sigma(0,j',k')_n\neq n.$ Hence $k'=n$. Therefore, $$\sigma(0,n-j,n-j+i)=\sigma(0,j+\bar{j},n)\circ \sigma(0,j',n)\in B.$$A contradiction since $i\neq j$. This proves the assertion.

(ii) Taking the inverse of both sides of the equation in (ii) gives by (\ref{eqa18oct})
$$\sigma(i,n-j+i,n)=\sigma(0,n-\bar{j},n)\circ \sigma(i',j',k')^{-1}.$$ Now, from (i), $n-\bar{j}=n-j+i$, and the assertion follows.
\end{proof}
\begin{prop}
\label{prova}
The bipartite graphs arising from the components of the partition in Lemma \ref{lem1oct9} have the following properties.
\begin{itemize}
\item[\rm(i)] In the bipartite subgraph $(L\cup F,B)$ of $\Cat$, every vertex in $L\cup F$ has degree $1$ while every vertex of $B$ has degree $n-2$.
\item[\rm(ii)] The bipartite subgraph $(L,F)$ of $\Cat$ is a $(1,1)$-biregular graph.
\end{itemize}
\end{prop}
\begin{proof} (i) Lemma \ref{leaoct19} (i) together with Lemma \ref{prop1oct12} (i) show that every vertex in $F$ has degree $1$. Corollary \ref{lem2oct11} ensures
that this holds true for $L$.

For every $1\le j \le n-1$, Lemma  \ref{prop1oct12} (iii) shows that there exist at least $j-1$ edges incident with $\sigma(0,j,n)$ and a vertex in $F$. Furthermore, from Lemma  \ref{prop1oct12} (iv), there exist at least $n-j-1$ edges incident with $\sigma(0,j,n)$ and a vertex in $L$. Therefore, at least $n-2$ edges incident with $\sigma(0,j,n)$ have a vertex in $L\cup F$. On the other hand, this number cannot exceed $n-2$ since $|L\cup F|=(n-1)(n-2)$ from Lemma \ref{lem1oct9}. This proves the first assertion.

(ii) From Lemma  \ref{prop1oct12} (ii), there exists at least one edge with a vertex in $F$ and another in $L$. Also, Lemma \ref{leaoct19} (ii) ensures the uniqueness of such an edge.
\end{proof}
From now on, $\bar{\Gamma}(W)$ stays for the induced subgraph of $\bar{\Gamma}$ on the vertex-set $W$.
 \begin{cor}
\label{cor3oct13ter}
$B$ is the unique maximal clique of $\bar{\Gamma}$ of size $n-1$ containing an edge of $\bar{\Gamma}(B)$.
\end{cor}
\begin{proof}
Proposition \ref{prova} (i) together with Lemma \ref{lem1oct9} show that the endpoints of an edge of $\bar{\Gamma}(B)$ do not have a common neighbor outside $B$.
\end{proof}
Computations performed by using the package ``grape'' of GAP \cite{gap} show that $\bar{\Gamma}$ is a $6$-regular subgraph for $n=5$ and $8$-regular subgraph for $n=6$, but $\bar{\Gamma}$ is only $3$-regular for $n=4$. This generalizes to the following result.
\begin{prop}\label{thm1oct11}
$\bar{\Gamma}$ is a $2(n-2)$-regular graph whenever $n\geq 5$.
\end{prop}
\begin{proof}
Since $B$ is a maximal clique of size $n-1$, every vertex of $B$ is incident with $n-2$ edges of $\bar{\Gamma}(B)$. From Proposition \ref{prova} (i), as many as $n-2$ edges incident with a vertex in $B$ have an endpoint in $L\cup F$. Thus, the assertion holds for the vertices in $B$.

In $\bar{\Gamma}(F)$ every vertex has degree $2(n-1)-4=2n-6,$ by induction on $n$. This together with Proposition \ref{prova} (ii) show that every vertex of $\bar{\Gamma}(F)$ has degree $2n-5$ in $\bar{\Gamma}(L\cup F)$. By Corollary \ref{lem2oct11}, this holds true for every vertex of $\bar{\Gamma}(L)$. The degree increases to $2n-4$ when we also count the unique edge in $\bar{\Gamma}(B),$ according to the first assertion of Proposition \ref{prova} (i).

In $\bar{\Gamma}(S_{n-2}^\vartriangle)$ every vertex has degree $2n-8,$ by induction on $n$. Furthermore, in $\bar{\Gamma}(L\cup S_{n-2}^\vartriangle)$ every vertex has degree $2n-6$ by induction on $n$, and the same holds for  $\bar{\Gamma}(F\cup S_{n-2}^\vartriangle)$. This together with Lemma \ref{lem1oct11} show that every vertex in $S_{n-2}^\vartriangle$ is the endpoint of exactly $2(2n-6)-(2n-8)$ edges in $\bar{\Gamma}$.
\end{proof}
Our next step is to determine the set of all maximal cliques of $\bar{\Gamma}$ of size $2$. From now on, we will be referring to the edges of the complete graph arising from a clique as the edges of the clique.
According to Lemma \ref{lem2oct13} (v), let $\Lambda$ be the set of all edges $$e_l=\{\sigma(l,l+1,l+3),\sigma(l,l+2,l+3)\},$$ where $l$ ranges over $\{0,1,\ldots n-3\}$. From (\ref{eqa18oct}), the endpoints of such an edge are the inverse of one another.
\begin{prop}
\label{thm2oct11} Let $n\geq 5$. The edges in $\Lambda$ together with three more edges
\begin{equation}
\label{eq1aoct12}
\begin{array}{lll}
e_{n-2}&=&\{\sigma(0,n-2,n-1),\sigma(0,n-2,n)\},\\
e_{n-1}&=&\{\sigma(1,n-1,n),\sigma(0,1,n-1)\},\\
e_{n}&=&\{\sigma(0,2,n),\sigma(1,2,n)\},\\
\end{array}
\end{equation}are pairwise disjoint edges of maximal cliques of $\bar{\Gamma}$ of size $2$.
\end{prop}
\begin{proof} Since $n\geq 5$, the above edges are pairwise disjoint.

Now, by (\ref{22march2015}), the following equations
\begin{equation}\label{eq1oct13}
\begin{array}{llllll}
\bar{\textsf{f}}(\sigma(l,l+1,l+3))&=&\sigma(l-1,l,l+2),& \mbox{ for } l\geq 1;\\
\bar{\textsf{f}}(\sigma(l,l+2,l+3))&=&\sigma(l-1,l+1,l+2) & \mbox{ for } l\geq 1; \\
\bar{\textsf{f}}(\sigma(0,1,3))&=&\sigma(0,2,n);\\
\bar{\textsf{f}}(\sigma(0,2,3))&=&\sigma(1,2,n);\\
\bar{\textsf{f}}(\sigma(0,2,n)&=&\sigma(1,n-1,n);\\
\bar{\textsf{f}}(\sigma(1,2,n))&=&\sigma(0,1,n-1);\\
\bar{\textsf{f}}(\sigma(1,n-1,n)&=&\sigma(0,n-2,n-1); \\
\bar{\textsf{f}}(\sigma(0,1,n-1))&=&\sigma(0,n-2,n);\\
\bar{\textsf{f}}(\sigma(0,n-2,n-1))&=&\sigma(n-3,n-2,n); \\
\bar{\textsf{f}}(\sigma(0,n-2,n))&=&\sigma(n-3,n-1,n).
\end{array}
\end{equation}
hold. This shows that $\bar{\textsf{f}}$ leaves the set $\Lambda\cup\{e_{n-2},e_{n-1},e_n\}$ invariant acting on it as the cycle permutation
$(e_n,\,e_{n-1},\cdots, e_1,\,e_0)$. 

Now, it suffices to verify that $e_n$ is a maximal clique of $\bar\Gamma$. Assume on the contrary that $\sigma=\sigma(i,j,k)$ is adjacent to both $\sigma(1,2,n)$ and $\sigma(0,2,n)$. As $\sigma(0,2,n)\in B$, Lemma \ref{lem1oct11} implies that $\sigma\in L\cup F$. Also, Proposition \ref{prova} (i) shows that $\sigma(0,2,n)$ has degree $n-2$ in $L\cup F$. In particular, in the proof of Proposition \ref{prova} (i), we have seen that $\sigma(0,2,n)$ must be adjacent to $n-3$ vertices of $L$, as $\sigma(1,2,n)\in F$. Then, by Lemma \ref{prop1oct12} (iv), $\sigma=\sigma(0,2,l)$ for some $l$ with $3\leq l < n$.

On the other hand, Proposition \ref{prova} (ii) shows that $\sigma\in L$ is uniquely determined by $\sigma(1,2,n)\in F$, and, by Lemma \ref{prop1oct12} (ii), $\sigma=\sigma(0,1,2)$, a contradiction.
\end{proof} 
From now on, $V$ denotes the set of the vertices of the edges $e_m$ with $m$ ranging over $\{0,1,\ldots,n\}$. For $n=4$, the edges $e_m$ are not pairwise disjoint, but computations show that they are also edges of maximal cliques of $\bar{\Gamma}$ of size $2$.
\begin{lem}
\label{lem1oct18} The toric maps and the reverse map preserve $V.$ Then, the toric-reverse group is regular on $V,$ and $\bar{\Gamma}(V)$ is a vertex-transitive graph.
\end{lem}
\begin{proof}
Since $\bar{\textsf{F}}$ is the subgroup generated by $\bar{\textsf{f}}$, from (\ref{eq1oct13}) follows that $\bar{\textsf{F}}$ preserves $V$ and has two orbits on $V$, each of them containing one of the two endpoints of the edges $e_m$ with $0\le m \le n$.

In addition, by (\ref{eq1oct11}), the reverse map $\textsf{g}$ interchanges the endpoints of $e_m$ with $0\le m\le n-3$ and $m=n-1$ while 
$$\begin{array}{lll}
\bar{\textsf{g}}(\sigma(0,n-2,n-1))&=&\sigma(1,2,n);\\
\bar{\textsf{g}}(\sigma(0,n-2,n)&=&\sigma(0,2,n);\\
\bar{\textsf{g}}(\sigma(1,2,n))&=&\sigma(0,n-2,n-1);\\
\bar{\textsf{g}}(\sigma(0,2,n))&=&\sigma(0,n-2,n).
\end{array}$$This implies that $\textsf{g}$ preserves $V$, and $\overline{\textsf{D}}_{n+1}$ acts transitively on $V$.
 
Now, since $|V|=2(n+1)$ and $\overline{\textsf{D}}_{n+1}$ has order $2(n+1)$, then $\overline{\textsf{D}}_{n+1}$ is regular on $V$.
\end{proof}
Our next step is to show that the $e_m$ with $0\le m \le n$ are the edges of all maximal cliques of $\bar{\Gamma}$ of size $2$. Computations performed by using the package ``grape'' of GAP \cite{gap} show that the assertion is true for $n =4,5,6$.
\begin{lem}\label{cor1oct13}
The edge of every maximal clique of $\bar{\Gamma}$ of size $2$ is one of the edges $e_m$ with $0\le m \le n$.
\end{lem}
\begin{proof}
On the contrary take an edge $e$ of a maximal clique of $\bar{\Gamma}$ of size $2$ other than the edges $e_m$. Since $L\cup S_{n-2}^\vartriangle\subset T_{n-1}$, by induction on $n\geq4$, $e$ is an edge of $\bar{\Gamma}(F\cup B)$. Also, $e$ has one endpoint in $B$ and the other in $F$, as $B$ is clique.

Now, let the endpoint of $e$ in $B$ be $\sigma(0,j,n)$ for some $1\leq j\leq n-1$. Then, by the proof of the first assertion of Proposition \ref{prova} (i), the vertex $\sigma(0,j,n)$ is adjacent to $\sigma(\bar{i},j,n)$ for any $0\le \bar{i} < j$. As the vertices $\sigma(\bar{i},j,n)$ for $0\le \bar{i} < j$ are adjacent by Lemma \ref{lem2oct13} (iii), $e$ is and edge of the triangle of vertices $\sigma(0,j,n)$, $\sigma(i',j,n),$ and $\sigma(\bar{i},j,n)$ with $i'\neq\bar{i}$ and $0\le i',\bar{i} < j$, a contradiction.
\end{proof}

Lemma \ref{cor1oct13} shows that $V$ consists of the endpoints of the edges of $\bar{\Gamma}$ which are the edges of maximal cliques of size $2$. Thus $\bar{\Gamma}(V)$ is relevant for the study of $\Cat$. We show some properties of $\bar{\Gamma}(V)$.
\begin{prop}\label{lemgoct12}
$\bar{\Gamma}(V)$ is a $3$-regular graph.
\end{prop}
\begin{proof} First we prove the assertion for the endpoint $v=\sigma(0,2,n)$ of $e_{n}$. By Lemma \ref{lem2oct13} (i) (iii) (v), $\sigma(0,2,3),\sigma(1,2,n),$ and $\sigma(0,n-2,n)$ are neighbors of $v$. Since $\sigma(1,2,n)\in F$ and $\sigma(0,2,3)\in L$, from the first assertion of Proposition \ref{prova} (i), $v\in B$ is not adjacent to any other vertex in either $V\cap F$ or $V\cap L$. Also, Lemma \ref{lem1oct11} yields that no vertex in $V\cap S_{n-2}^\vartriangle$ is adjacent to $\sigma(0,2,n)$. Thus, $v$ has degree $3$ in $\bar{\Gamma}(V)$.
 
Now the claim follows from Lemma \ref{lem1oct18}.
\end{proof}
\begin{rem} By a famous conjecture of Lov\'asz, every finite, connected, and  vertex-transitive graph contains a Hamiltonian cycle, except the five known counterexamples; see \cite{LL,BL}. Then, the second assertion of Lemma \ref{lem1oct18} and Proposition \ref{thm1oct13} show that the Lov\'asz conjecture holds for the graph $\bar\Gamma(V)$.
\end{rem}
\begin{prop}\label{thm1oct13}
$\bar{\Gamma}(V)$ is a Hamiltonian graph whenever $n\geq 5$.
\end{prop}
\begin{proof}
Let $v_1=\sigma(n-4,n-3,n-1),\quad v_2=\sigma(n-4,n-2,n-1)$ be the endpoints of $e_{n-4}$. We start by exhibiting a path $\mathcal{P}$ in $V$ beginning with $\sigma(0,2,3)$ and ending with $v_1$ that visits all vertices $\sigma(l,l+1,l+3),\sigma(l,l+2,l+3)\in\Lambda$ with $0\leq l\leq n-4$.

For $n=5$, $v_1=\sigma(1,2,4)$, and $$\mathcal{P}=\sigma(0,2,3),\sigma(0,1,3),\sigma(1,3,4),v_1.$$

Assume $n>5$. For every $l$ with $0\leq l\leq n-4$, Lemma \ref{lem2oct13} (ii) (v) show that both edges below are incident to $\sigma(l,l+1,l+3)$:
$$\{\sigma(l,l+1,l+3),\sigma(l+1,l+3,l+4)\},\quad\{\sigma(l,l+2,l+3),\sigma(l,l+1,l+3)\}.$$Therefore,
$$\begin{array}{ll}
\sigma(0,2,3),\sigma(0,1,3),\sigma(1,3,4),\ldots,\sigma(l,l+2,l+3),\sigma(l,l+1,l+3),\\\sigma(l+1,l+3,l+4),\ldots,\,v_1
\end{array}$$is a path $\mathcal{P}$ with the requested property.

By Lemma \ref{lem2oct13}, there also exists a path $\mathcal{P'}$ beginning with $v_1$ and ending with $\sigma(0,2,3)$ which visits the other vertices of $V,$ namely
$$\begin{array}{lll}
v_1,\sigma(n-3,n-1,n),\sigma(n-3,n-2,n),\sigma(0,n-2,n),\sigma(0,n-2,n-1),\\\sigma(0,1,n-1),\sigma(1,n-1,n),\sigma(1,2,n),
\sigma(0,2,n),\sigma(0,2,3).
\end{array}$$By Theorem \ref{thm2oct11}, the vertices are all pairwise distinct. Therefore the union of $\mathcal{P}$ and $\mathcal{P'}$ is a cycle in $V$ that visits all vertices. This completes the proof.
\end{proof}
\begin{rem}
For $n\geq 4$, by Proposition \ref{lemgoct12} and Theorem \ref{teor2}, Proposition \ref{thm1oct13} also follows from a result of Alspach and Zhang \cite{AZ} who proved that all cubic Cayley graphs on dihedral groups have Hamilton cycles.
\end{rem}

\section{The automorphism group of the \\block transposition graph}\label{s6}
We are in a position to give a proof for Theorem \ref{teor2}. Since $\Gamma\cong \bar\Gamma$ and $\textsf{D}_{n+1}\cong\overline{\textsf{D}}_{n+1}$, we may prove Theorem \ref{teor2} using the right-invariant notation.

From Proposition \ref{22march22D2015}, the toric-reverse group $\overline{\textsf{D}}_{n+1}$ is a subgroup $\bar{\textsf{R}}$, the automorphism group of $\bar{\Gamma}$. Also, $\overline{\textsf{D}}_{n+1}$ is regular on $V$, by the second assertion of Lemma \ref{lem1oct18}. Therefore, Theorem \ref{teor2} is a corollary of the following lemma.
\begin{lem}
\label{lemfoct12} The identity is the only automorphism of $\bar{\Gamma}$ fixing a vertex of $V$ whenever $n\geq 5$.
\end{lem}
\begin{proof} We prove the assertion by induction on $n$. Computation shows that the assertion is true for $n=5,6$. Therefore, we assume $n\geq 7$.

First we prove that any automorphism of $\bar{\Gamma}$ fixing a vertex $v\in V$ is an automorphism of $\bar{\Gamma}(V)$ as well. Since $\overline{\textsf{D}}_{n+1}$ is regular on $V$, we may
limit ourselves to take $\sigma(0,2,n)$ for $v$. Let $\bar{\textsf{H}}$ be the subgroup of $\bar{\textsf{R}}$ which fixes $\sigma(0,2,n)$.

We look inside the action of $\bar{\textsf{H}}$ on $\bar{\Gamma}(V)$ and show that $\bar{\textsf{H}}$ fixes the edge $\{\sigma(0,2,n),\sigma(0,n-2,n)\}$. By Proposition \ref{lemgoct12}, $\bar{\Gamma}(V)$ is $3$-regular. More precisely, the endpoints of the edges of $\bar{\Gamma}(V)$ which are incident with $\sigma(0,2,n)$ are $\sigma(0,2,3),\,\sigma(1,2,n),$ and $\sigma(0,n-2,n)$; see Lemma \ref{prop1oct12} (i) (iii) (v). Also, by Proposition \ref{thm2oct11}, the edge $e_{n-1}=\{\sigma(0,2,n),\sigma(1,2,n)\}$ is the edge of a maximal clique of $\bar{\Gamma}$ of size $2$, and no two distinct edges of maximal cliques of $\bar{\Gamma}$ of size $2$ have a common vertex. Thus, $\bar{\textsf{H}}$ fixes $\sigma(1,2,n)$. Now, from Corollary \ref{cor3oct13ter}, the edge $\{\sigma(0,2,n),\sigma(0,n-2,n)\}$ lies in a unique maximal clique of size $n-1$. By Lemma \ref{lem2oct13} (i), the edge $\{\sigma(0,2,n),\sigma(0,2,3)\}$ lies on a clique of size $n-2$ whose set of vertices is $\{\sigma(0,2,k)| 3\le k\le n\}$. Here, we prove that any clique $C$ of size $n-2$ containing the edge $\{\sigma(0,2,n),\sigma(0,2,3)\}$ is maximal. By the first assertion of Proposition \ref{prova} (i), $\sigma(0,2,3)$ is adjacent to a unique vertex in $B$, namely $\sigma(0,2,n)$. On the other hand, among the $2(n-2)$ neighbors of $\sigma(0,2,n)$ off $V$, only as many as $n-3$ vertices are off $B\cap V$, by the proof of Proposition \ref{thm1oct11}. Then, $C$ does not extend to a clique of size $n-1$. Therefore, $\bar{\textsf{H}}$ cannot interchange the edges $\{\sigma(0,2,n),\sigma(0,n-2,n)\}$ and $\{\sigma(0,2,n),\sigma(0,2,3)\}$ but fixes both.

Also, by Proposition \ref{lemgoct12} and Lemma \ref{prop1oct12} (i) (iii), $\sigma(0,n-2,n)$ is adjacent to $\sigma(0,n-2,n-1)$ and $\sigma(n-3,n-2,n).$ Since $e_{n-2}$ is the edge of a maximal clique of $\bar{\Gamma}$ of size $2$, $\bar{\textsf{H}}$ fixes $e_{n-2}=\{\sigma(0,n-2,n-1),\sigma(0,n-2,n)\}$. This together with what we have proven so far shows that $\bar{\textsf{H}}$ fixes $\sigma(n-3,n-2,n),$ and then the edges $e_{n-3}=\{\sigma(n-3,n-2,n),\sigma(n-3,n-1,n)\}$.

Now, as the edge $\{\sigma(0,2,n),\sigma(0,n-2,n)\}$ is in $\bar{\Gamma}(B)$, Corollary \ref{cor3oct13ter} implies that
$\bar{\textsf{H}}$ preserves $B$. And, as $\bar{\textsf{H}}$ fixes $\{\sigma(0,2,n),\sigma(0,2,3)\}$, $\bar{\textsf{H}}$  must fix $\sigma(0,2,n)\in B$ and $\sigma(0,2,3)\notin B$. Also, $e_0=\{\sigma(0,1,3),\sigma(0,2,3)\}$ is preserved by $\bar{\textsf{H}}$, as we have seen above. Therefore,  $\sigma(0,1,3)$ is also fixed by $\bar{\textsf{H}}$. Furthermore, $\sigma(2,3,5)\in S_{n-2}^\vartriangle$ is adjacent to $\sigma(0,2,3)$ in $\bar{\Gamma}(V)$, by Proposition \ref{lemgoct12} and Lemma \ref{prop1oct12} (ii); and then it is fixed by $\bar{\textsf{H}}$, as $\bar{\textsf{H}}$ preserves $S_{n-2}^\vartriangle,$ by Lemma \ref{lem1oct9}. Therefore, we have that $\bar{\textsf{H}}$ induces an automorphism group of $\bar{\Gamma}(S_{n-2}^\vartriangle)$ fixing a vertex $\sigma(2,3,5)\in S_{n-2}^\vartriangle$. Then $\bar{\textsf{H}}$ fixes every block transpositions in $S_{n-2}^\vartriangle\cong T_{n-2},$ by the inductive hypothesis. In particular, $\bar{\textsf{H}}$ fixes all the vertices in $V \cap S_{n-2}^\vartriangle,$ namely all vertices in $\Lambda$ belonging to $e_l$ with $0<l<n-3$.

This together with what proven so far shows that $\bar{\textsf{H}}$ fixes all vertices of $V$ with only two possible exceptions, namely the endpoints of the edge $e_{n-1}=\{\sigma(1,n-1,n),\sigma(0,1,n-1)\}$. In this exceptional case, $\bar{\textsf{H}}$ would swap $\sigma(0,1,n-1)$ and $\sigma(1,n-1,n).$ Actually, this exception cannot occur since $\sigma(0,1,n-1)$ and $\sigma(1,n-1,n)$ do not have a common neighbor, and $\bar{\textsf{H}}$ fixes their neighbors in $V$. Therefore, $\bar{\textsf{H}}$ fixes every vertex in $V$. Hence, $\bar{\textsf{H}}$ is the kernel of the permutation representation of $\bar{\textsf{R}}$ on $V$. Thus $\bar{\textsf{H}}$ is a normal subgroup of $\bar{\textsf{R}}$.

Our final step is to show that the block transpositions in $L\cup B$ are also fixed by $\bar{\textsf{H}}$. Take any block transposition $\sigma(0,j,k)$. Then the toric class of $\sigma(0,j,k)$ contains a block transposition $\sigma(i',j',k')$ from $S_{n-2}^\vartriangle$. This is a consequence of the equations below which are obtained by using (\ref{22march2015})
\begin{equation}\label{eqa14oct}
\begin{array}{llll}
{\bar{\textsf{f}}}^{\,2}(\sigma(0,j,k))&=&\sigma(j-2,k-2,n-1),& j\geq 3; \\
{\bar{\textsf{f}}}^{\,3}(\sigma(0,1,k))&=&\sigma(k-3,n-2,n-1),&k\geq 4;\\
{\bar{\textsf{f}}}^{\,4}(\sigma(0,1,2))&=&\sigma(n-3,n-2,n-1);&{}\\
{\bar{\textsf{f}}}^{\,5}(\sigma(0,1,3))&=&\sigma(n-4,n-3,n-1);&{}\\
{\bar{\textsf{f}}}^{\,4}(\sigma(0,2,k))&=&\sigma(k-4,n-3,n-1),& k\geq 5;\\
{\bar{\textsf{f}}}^{\,5}(\sigma(0,2,3))&=&\sigma(n-4,n-2,n-1);&{}\\
{\bar{\textsf{f}}}^{\,6}(\sigma(0,2,4))&=&\sigma(n-5,n-3,n-1).&{}\\
\end{array}
\end{equation}Since $\sigma(i',j',k')\in S_{n-2}^\vartriangle$ we know that $\bar{\textsf{H}}$ fixes $\sigma(i',j',k')$. From this we infer that $\bar{\textsf{H}}$ also fixes $\sigma(0,j,k)$. In fact, as $\sigma(0,j,k)$ and $\sigma(i',j',k')$ are torically equivalent, $\bar{\textsf{u}}(\sigma(i',j',k'))=\sigma(0,j,k)$ for some $\bar{\textsf{u}}\in\bar{\textsf{F}}$. Take any $\bar{\textsf{h}}\in\bar{\textsf{H}}$. As $\bar{\textsf{H}}$ is a normal subgroup of $\bar{\textsf{R}}$, there exists $\bar{\textsf{h}}_1\in \bar{\textsf{H}}$ such that ${\bar{\textsf{u}}}\circ\bar{\textsf{h}}_1=\bar{\textsf{h}}\circ \bar{\textsf{u}}$. Hence
$$\sigma(0,j,k)={\bar{\textsf{u}}}(\sigma(i',j',k'))=
{\bar{\textsf{u}}}\circ\bar{\textsf{h}}_1(\sigma(i',j',k'))=\bar{\textsf{h}}\circ \bar{\textsf{u}}(\sigma(i',j,',k'))$$
whence $\sigma(0,j,k)=\bar{\textsf{h}}(\sigma(0,j,k)).$ Therefore, $\bar{\textsf{H}}$ fixes every block transposition in $L\cup B.$ 

Also, this holds true for $F$, by the second assertion of Proposition \ref{prova}. Thus, by Lemma \ref{lem1oct9}, $\bar{\textsf{H}}$ fixes
every block transposition. This completes the proof.
\end{proof}
\begin{rem}
Lemma \ref{lemfoct12} yields Theorem \ref{teor2} for $n\geq 5$. For $n=4$, computations performed by using the package ``grape'' of GAP \cite{gap} show that Theorem \ref{teor2} is also true.
\end{rem}

\end{document}